\documentclass[10pt]{amsart}
\usepackage{amssymb,amsmath,amsfonts,mathrsfs}

\newtheorem{theorem}{Theorem}[section]

\theoremstyle{definition}

\theoremstyle{remark}

\title{The Heat Equation on the Finite Poincar\'e Upper Half-Plane}
\author{M. R. DeDeo and Elinor Velasquez}
\address{Department of Mathematics and Statistics\\
        University of North Florida\\
         Jacksonville, FL 32224,  USA
            }
\email{mdedeo@unf.edu}
\address{Department of Mathematics\\
        University of California, Berkeley\\
         Berkeley, CA 94720,  USA
            }
\email{evelas21@mail.ccsf.edu}
\begin{document}
\begin{abstract}
A differential-difference operator is used to model the heat equation on a finite graph analogue of Poincar\'e's upper half-plane.
Finite analogues of the classical theta functions are shown to be solutions to the heat equation in this setting.
\end{abstract}
\maketitle
\section{Introduction}
Consider a finite graph $\varGamma$ with an arbitrary vertex set $V$ and some distinguished vertex $v_0 \in V$. We seek to determine the heat flow from vertex to vertex over continuous time. With regard to a specific initial condition, a solution to the above problem is modeled by the fundamental solution to the following set of equations:
\begin{align}\label{eq-1}
&\left[
\frac{\partial}{\partial t} - \Delta
\right]u(v,t) = 0,\quad t>0,\nonumber\\
&\lim_{t\to 0+}
\sum_{v \in V} u(v,t) f(v) = f(v_0),
\end{align}
for a fixed function $f$ on $V$, and $\Delta$ the combinatorial
Laplacian defined on the graph $\varGamma$. In this paper, we compute an explicit solution for the heat equation, with initial data concentrated at a point on $\varGamma$, a graph which has been constructed to act as a finite analogue of the Poincar\'e upper half-plane.
Specifically, $\varGamma$ is a finite Cayley graph associated with the group of rank 2 invertible matrices over a finite field of characteristic $q$, with edges created via the generating set of this group.

Solving for the fundamental solution to the heat equation on a bounded domain is a classical problem in partial differential equations. When the domain is the circle, for instance, the fundamental solution of the heat equation can be described by a theta function.

Solving the heat equation on $H$, the Poincar\'e upper half-plane amounts to computing the temperature function over time on an insulated, thin hyperbolic triangle \cite{Ter1}. This hyperbolic triangle is non-compact, yet analogues to classical theta functions occur when describing the fundamental solution.
In this paper we determine domains in which theta functions occur as fundamental solutions to the heat equation.

The fundamental solution to the heat equation has applications to random walks on groups in probability theory. See Grigorchuk \cite{Gri}, Levit-Molchanov \cite{LM}, and Pagliacci \cite{Pag} for some recent results. In the work of Pagliacci, the fundamental solution takes the form of a time-dependent inverse problem over discrete variables where $\varGamma$ is a finite graph and a random walk takes place on $\varGamma$.

Begin with the initial position at $v_0 \in$ $V$ which can be viewed as the initial probability distribution $P_0(v)$ for some $v\in V$. The probability distribution after $n$ time-steps, $P_n(v)$, is completely determined by the initial state $P_0$ and the transition-matrix $A$, the adjacency matrix of the graph. Since the adjacency matrix is related to the combinatorial Laplacian, we obtain a set of difference only equations representing the linear heat equation on $\varGamma$. In this case, $\varGamma$ is a homogeneous tree of order $q + 1$.

Suppose $f: \varGamma \times \mathbb{N} \rightarrow \mathbb{C}$, such that $f(x,k) = f_k(x)$. Then the heat equation on $\varGamma$ described by Pagliacci is given by
\begin{equation}\label{pagl}
f_k(x) = \sum_{d(x,y)=1} p(x,y) f_{k-1}(y),
\end{equation}
with initial condition $f_0 (x)$, $d$ the canonical distance function on $\varGamma$, and $p$ the transition probability on $V$.

In other words, Pagliacci defines the heat equation on $\varGamma$ using a weighted adjacency matrix for the Laplacian operator and a finite difference operator over \(\mathbb{Z}\) for the time differential. There is no initial condition such as in (\ref{eq-1}).

Not all investigations of the heat equation on a finite graph use a combinatorial Laplacian. In the work of Gavean, Okada, and Okada \cite{GOO}, graphs of both finite and infinite vertex sets are considered, but the viewpoint is quite different. The functions are defined on the vertex set of $\varGamma$ as well as along the continuous segments of the graph, namely the edges. Thus the graphs considered in \cite{GOO} are equipped with a Riemannian structure and the Laplacian is induced from this structure. In other words, in \cite{GOO} the operator
$\frac{\partial}{\partial t} - {\alpha}_j \frac{{\partial}^2}{\partial x^2},$
for $j \in \mathbb{N} $, is applied to functions on the graph $\varGamma$. For the results in this paper, we do not assume the functions to be evaluated on the edges of the graphs as in \cite{GOO}.

This paper discusses the solutions connected to differential-difference operators of the form $\frac{\partial}{\partial t} - \Delta$ with $\Delta$ the combinatorial Laplacian. We believe this to be the first study of the heat equation via a differential-difference operator on a graph equipped with a graph-theoretic Laplacian operator. In addition, the domain is rich with structure, making the results nontrivial.
Our technique for solution is an extension of the method of separation of variables on a non-abelian Cayley graph which involves special functions.

Classically, the heat equation on bounded domains, specifically on tori, yield theta functions as solutions. We study a finite analogue of the Poincar\'e upper half-plane, namely the finite upper half-plane introduced by Terras \cite{Ter2}. Using this case, we investigate the periodicity inherent in the solutions of bounded domains.

The solutions involve zonal spherical functions which come with a natural periodicity. In addition, the related theta functions are automorphic forms and the resultant periodicity interweaves representation theory with the heat equation. We hope this paper stimulates more study of this interplay.

\section{Finite upper half-planes}
Fix $q \ge 3$ odd and $F_q$, a finite field of $q$ elements. Suppose
$\delta \in F_q$ is not a square of any element of $F_q$. Then the finite upper half-plane is said to be the set $$ H_q = \{ x + y\sqrt{\delta} \mid x \in F_q, y \in F_q^{\times} \}, $$ with $F_q^{\times}$ the multiplicative group of $F_q$ \cite{Ter2}.
Recall that the affine group is the subgroup

$$Aff(q) = \left\{ \begin{pmatrix} y&x\\ 0&1 \end{pmatrix}\right|(y,x) \in F_q^{\times} \times F_q\biggr\}$$

of $GL(2,F_q)$. Form a Cayley graph, denoted
${\varGamma}_q (H_q, H_q \times S)$, by taking $H_q$ as the vertex set and $H_q \times S$ as the edge set with $S$ the generating set of the affine group.
Then, for $r \in F_q$,

$$
S_r = \left\{ \begin{pmatrix} y&x\\ 0&1 \end{pmatrix} \right| x^2 = ry + \delta (y - 1)^2, y \in F_q^{\times}, x \in F_q  \biggr\}. $$

Thus the constructed Cayley graph depends on which $S_r$, $r \in F_q$, is used to generate the edge set \cite{CPTTV}. If $e$ is an edge in the graph, then $e$ has $e_-$, $e_+$ as an initial and a terminal vertex, respectively. So the edge set $H_q \times S_r$ means that $e_- = z \in H_q$ and $e_+ = zs, s \in S_r$. With generating set $S_r$, the graph is connected.

Equip ${\varGamma}_q (H_q, H_q \times S)$ with the combinatorial Laplacian. Suppose $A$ is the adjacency matrix with entries $A_{z,w}$
= number of paths from vertex $z$ to vertex $w$, $z,w \in H_q$, and $I$ is the identity matrix of size $q(q-1)$. If $f \in C(H_q)$, the set of functions on $H_q$, the operator $\Delta :C(H_q) \rightarrow C(H_q)$, defined by $$ \Delta f(x) = ((q+1)I -A)f(x), $$ is said to be the combinatorial Laplacian operator associated to ${\varGamma}_q (H_q, H_q \times S)$.

We require that the fundamental solution on $H_q$, $E(t;x,y_0)$,  satisfy the following set of equations:

\begin{equation}\label{eq-2}
\Delta E(t;x,y_0) -
\frac{\partial}{\partial t} E(t;x,y_0) = 0,
\quad t>0,
\end{equation}
\begin{equation}\label{eq-3}
\lim_{t\to 0+}\frac1{q(q-1)} \sum_{x \in H_q} E(t;x,y_0) f(x)= f(y_0),
\end{equation}
following the notation of Benabdallah \cite{Ben} and Fischer, Jungster and Williams \cite{FJW}. The solution $E(t;x,y_0)$ is a function on $\mathbb{R} \times H_q \times \{ y_0 \}$. In other words, time $t$ and space $x$ are taken to be continuous and discrete variables, respectively, with the initial data concentrated at a point $y_0 \in H_q$.
\section{Method of images}
When $H$ is the Poincar\'e upper half-plane, an explicit solution to the heat equation with the initial condition $E(0,x) = f(x)$ is computed using the method of images \cite{Ter1}. Benabdallah \cite{Ben} uses the method of images to solve the heat equation, with an initial condition analogous to (\ref{eq-3}) on compact, homogeneous spaces. We adapt Benabdallah's technique to solve (\ref{eq-2}) -- (\ref{eq-3}).

Let ${\varGamma}_q (GL(2,F_q),GL(2,F_q) \times S_{GL(2,F_q)})$ denote the Cayley graph formed by taking $GL(2,F_q)$ as the vertex set and $GL(2,F_q) \times S_{GL(2,F_q)}$ as the edge set with $S_{GL(2,F_q)}$ the generating set of $GL(2,F_q)$.
Equip ${\varGamma}_q (GL(2,F_q),GL(2,F_q) \times S_{GL(2,F_q)})$ with the combinatorial Laplacian, ${\Delta}_{GL(2,F_q)}$.

Suppose $\delta$ is a non-square in $F_q$. Let

$$ K =\left\{ \begin{pmatrix} a&\delta b\\ b&a \end{pmatrix} \right| a^2 - \delta b^2 = 1  \biggr\}. $$

Then $H_q = GL(2,F_q)/K$ is a homogeneous space with $GL(2,F_q)$ acting transitively on the points of $H_q$. If $\pi : GL(2,F_q) \rightarrow GL(2,F_q)/K$ and $f$ is a function on $H_q$, then $\Delta f \circ \pi = {\Delta}_{GL(2,F_q)} (f \circ \pi )$ with $\Delta$ and ${\Delta}_{GL(2,F_q)}$ the combinatorial Laplacians associated to ${\varGamma}_q (H_q, H_q \times S)$ and ${\varGamma}_q (GL(2,F_q),GL(2,F_q) \times S_{GL(2,F_q)})$, respectively.

\begin{theorem}\label{Thm 4.1}
Suppose that $E_{GL(2,F_q)}(t;z) = E_{GL(2,F_q)}(t;z,e)$ denotes the fundamental solution for the heat equation with initial concentration at the identity $e$ on $GL(2,F_q)$ and ${\Delta}_{GL(2,F_q)}$. Then $$ E(t;zK) = \frac1{q^2 - 1} \sum_{k \in K} E_{GL(2,F_q)} (t;zk), $$ is the fundamental solution to (\ref{eq-2})--(\ref{eq-3}) with initial concentration at the identity coset $K$. In other words, $E(t;zK) = E(t;zK,K)$.
\end{theorem}

\begin{proof}
The combinatorial Laplacian ${\Delta}_{GL(2,F_q)}$ is a left-invariant operator, analogous to the geometrical Laplacian. Let $f$ be a function on $GL(2,F_q)$. Left translation on $f$ is defined as $f^{g_1}
(x) = f(g_1,x)$ with $g_1, x \in GL(2,F_q)$.

Then $$
\aligned
{\Delta}_{GL(2,F_q)} f^{g_1} (x) &= ((q+1)I -
A)f^{g_1} (x) \\
&= (q+1)f^{g_1} (x) - \sum_{s\in Kg_0K = S}
f^{g_1} (xs) \\
&= (q+1)f(g_1x) - \sum_{s\in S} f(g_1xs)
\\
&= ({\Delta}_{GL(2,F_q)} f)^{g_1} (x).
\endaligned
$$

We see that if $E_{GL(2,F_q)}(t;x,y_0)$ is the fundamental solution for the heat equation on $GL(2,F_q)$ and ${\Delta}_{GL(2,F_q)}$, then $E_{GL(2,F_q)}(t;x,y_0)$ is invariant under left translations. This means that $E_{GL(2,F_q)}(t;gx,e) = E_{GL(2,F_q)}(t;x,g^{-1})$ for $e$ the identity in $GL(2,F_q)$. Thus $E_{GL(2,F_q)}(t;x,y_0)$ is invariant under the inner automorphisms of $GL(2,F_q)$.

We have a similar result for the fundamental solution on $H_q$: Since $H_q = GL(2,F_q)/K$, we have that $$ \aligned E(t;gxK,gy_0K) &= E(t;xK,y_0K) \\ &= E(t;y_0^{-1}xK,K) \\ &= E(t;y_0^{-1}xK). \endaligned $$

Consider the function $\tilde E(t;zK)$ defined by $$ \aligned \tilde E(t;zK) &= \frac1{q^2-1} \sum_{k\in K} E_{GL(2,F_q)} (t;zk) \\ &= \frac1{q^2-1} \sum_{k\in K} E_{GL(2,F_q)} (t;z,k^{-1}). \endaligned $$
Then, $$ \aligned [\Delta \tilde E] (gK) &=\biggl[ {\Delta}_{GL(2,F_q)} \biggl( \frac1{q^2-1} \sum_{k\in K} E_{GL(2,F_q)} (t;\cdot ,k^{-1}) \biggr) \biggr] (g) \\ &=\biggl[ \frac1{q^2-1} \sum_{k\in K} {\Delta}_{GL(2,F_q)} E_{GL(2,F_q)} (t;\cdot ,k^{-1}) \biggr] (g) \\ &=\biggl[ \frac1{q^2-1} \sum_{k\in K} -\frac{\partial}{\partial t} E_{GL(2,F_q)} (t;\cdot ,k^{-1}) \biggr] (g) \\ &= \frac{\partial}{\partial t} \biggl[ \frac1{q^2-1} \sum_{k\in K} E_{GL(2,F_q)} (t;\cdot ,k^{-1}) \biggr] (g) \\ &=\biggl[ -\frac{\partial}{\partial t} \tilde E   \biggr] (gK).
\endaligned
$$

Additionally, for all functions $f$ defined on $H_q$,
$$
\aligned
\frac1{q(q-1)}\sum_{H_q} \tilde E (t;gK)f(gK)
&=
\frac1{q(q-1)}\sum_{H_q} f(gK)
\biggl[
\frac1{q^2-1}\sum_K E_{GL(2,F_q)} (t;gk)
\biggr]
\\
&=
\frac1{q(q-1)^2(q+1)}\sum_{GL(2,F_q)} f\circ
\pi (g)E_{GL(2,F_q)} (t;g).
\endaligned
$$
In addition,
$$
\lim_{t \to 0+}
\frac1{q(q-1)^2(q+1)} \sum_{GL(2,F_q)}
E_{GL(2,F_q)} (t;g) f\circ \pi (g)
=
f\circ \pi (e)
=
f(K).
$$

Thus $\tilde E (t;gK)$ = $E(t;gK)$ =
      $\frac1{q^2-1} \sum_{k\in K}
E_{GL(2,F_q)} (t;gk)$.
\end{proof}

Suppose $\{ {\pi}^{\alpha} \mid \alpha \in \widehat{GL(2,F_q)} \}$, with $\widehat{GL(2,F_q)}$ the dual of $GL(2,F_q)$, is the set of equivalence classes of all unitary, irreducible representations of
$GL(2,F_q)$. Let $\{ {\chi}^{\alpha} \mid \alpha \in
\widehat{GL(2,F_q)} \}$ be the corresponding character set. $E_{GL(2,F_q)}$ invariant under the inner automorphisms of $GL(2,F_q)$ implies that $E_{GL(2,F_q)}(t;g)$ is a class function. Therefore, by the Peter-Weyl theorem, we have that
$$
E_{GL(2,F_q)} (t;g)
=
\sum_{\alpha \in \widehat{GL(2,F_q)}}
\langle {\chi}^{\alpha} \mid E(t; \cdot ) \rangle
{\chi}^{\alpha} (g)
$$ with $\langle \cdot,\cdot \rangle$ the inner product on  $GL(2,F_q)$ and Haar measure $dg$ = $\frac1{q(q-1)^2(q+1)}$.
Define $$
a_{\alpha} (t) = \frac1{q(q-1)^2(q+1)}
\sum_{GL(2,F_q)} E(t;g)
                 {\chi}^{\alpha} (g^{-1}).
$$

In other words, $a_{\alpha}$ is a Fourier coefficient, or rather, it is the Fourier transformed heat kernel $E(t;g)$, for $\alpha \in \widehat{GL(2,F_q)}$.

By Theorem~\ref{Thm 4.1},

\begin{align}\label{eq-4-1}
E(t;gK) &= \frac1{q^2-1}
          \sum_{k \in K} E(t;gk) \nonumber\\
        &=\frac1{q^2-1} \sum_{\alpha
\in \widehat{GL(2,F_q)}}
          \sum_{k \in K} a_{\alpha} (t)
{\chi}^{\alpha} (gk).
\end{align}

But $E_{GL(2,F_q)}(t;g)$ being a class function means that we can apply
$\pi :GL(2,F_q)$$\rightarrow H_q$ to show that $E(t;gK) \in C(K\backslash GL(2,F_q)/K)$, the set of $K$-bi-invariant functions on $GL(2,F_q)$. Therefore, $E(t;gK)$ has an expansion in terms of zonal spherical functions. 

Specifically, let $\widehat{GL(2,F_q)}_K$ denote the equivalence classes of irreducible, unitary representations of $GL(2,F_q)$ which are class 1 with respect to $K$. Let ${\omega}^{\alpha}$ denote the zonal spherical function associated to
$\alpha \in \widehat{GL(2,F_q)}_K$. Then

\begin{align}\label{eq-4-2}
E(t;gK) &= \sum_{\alpha \in
\widehat{GL(2,F_q)}_K}
          \langle{\omega}^{\alpha} \mid
E(t;\cdot )\rangle {\omega}^{\alpha}(gK)\nonumber\\
        &= \frac1{q^2-1} \sum_{\alpha
\in \widehat{GL(2,F_q)}_K}
          \sum_{k \in K}
\langle{\omega}^{\alpha} \mid E(t;\cdot )\rangle
          {\omega}^{\alpha} (gk).
\end{align}

Comparison with (\ref{eq-4-1}) and (\ref{eq-4-2}) permits the result:
$$
E(t;gK) = \sum_{\alpha \in
\widehat{GL(2,F_q)}_K} a_{\alpha} (t)
          {\omega}^{\alpha} (gK).
$$

Apply the Fourier transform with respect to $GL(2,F_q)/K$
to~(\ref{eq-2}).
The result is that the heat kernel
$$
a_{\alpha} (t) = k_{\alpha}
e^{-{\lambda}_{\alpha}t},
$$
with ${\lambda}_{\alpha}$ an eigenvalue of
${\Delta}_{GL(2,F_q)/K}$
= $\Delta$ and
$$
\aligned
k_{\alpha}
           &=
a_{\alpha} (0)
            =
\lim_{t\to 0+}\frac1{q(q-1)^2(q+1)} \sum_{GL(2,F_q)} E(t;g)
{\chi}^{\alpha} (g^{-1})
\\
           &=
{\chi}^{\alpha} (e) = d_{\alpha}
\endaligned
$$ where $d_{\alpha}$ denotes the dimension of the representation
${\pi}^{\alpha}$, $\alpha \in
\widehat{GL(2,F_q)}_K$. Now we only need to expand $E(t;gK)$ in terms of the class 1 representations.
Since $E(t;gK) \in C(K\backslash GL(2,F_q)/K)$,
$$
e(t;gK) = E(t;KgK) = \sum_{\alpha \in
\widehat{GL(2,F_q)}_K}
                     d_{\alpha}
e^{-{\lambda}_{\alpha}t}
                     {\omega}^{\alpha} (KgK).
$$
In other words,
$$
E(t;r) = \sum_{\alpha \in
\widehat{GL(2,F_q)}_K}
         d_{\alpha} e^{-{\lambda}_{\alpha}t}
         {\omega}^{\alpha} (r),
$$ with $r$ the radius of an orbit or $K$-double coset in $H_q$.

Zonal spherical functions of $H_q$ split into two types: those associated to the principal series representations of $GL(2,F_q)$ and those associated to the cuspidal series representations of $GL(2,F_q)$. Let $U = \{ z \in F_q (\sqrt{\delta}) \mid z\bar z =1\}$.

Define $\epsilon$, ${\nu}_0$ to be the sign characters of $F_q^{\times}$ and $U$, respectively, equal to 1 on squares and to -1 on nonsquares of $F_q^{\times}$ and $U$, respectively. Define $N(\alpha ) = \alpha \overline{\alpha}$, $Tr(\alpha ) = \alpha + \overline{\alpha}$ as the norm and trace of $\alpha \in F_q(\sqrt{\delta})$, respectively. Recall that $S_r$, fixed $r \in F_q$, denotes the generating set of the affine group.

\begin{theorem}\nonumber (\cite{GaZ})
The zonal spherical functions of $H_q$ associated to the principal series of $GL(2,F_q)$ are the functions
$$
\aligned
&{\omega}^{\beta} (0) = 1 \\
&{\omega}^{\beta} (\infty ) = \beta
(-1) \\
&{\omega}^{\beta} (r) = \frac1{q+1}
\sum_{z\in S_r} \beta (I(z)),
\endaligned
$$
with $\beta \in \widehat{F_q^{\times}}$, $r
\in F_q - \{ 1 \}$,
and if $z = x + y\sqrt{\delta} \in H_q$, then
$I(z) = y$.
\end{theorem}
\begin{theorem}(\cite{SA})
The zonal spherical functions of $H_q$ associated to the cuspidal representations of $GL(2,F_q)$ are the functions
$$
\aligned
&{\omega}^{\nu} (0) = 1 \\
&{\omega}^{\nu} (\infty ) = - \omega
(-1) \\
&{\omega}^{\nu} (r) = \frac1{q+1}
\sum_{u \in U}
                      \epsilon
\biggl(
Tr
\biggl(
u - \frac{1+r}{1-r}
\biggr)
\biggr)
                      {\nu}_0 (u) \nu (u),
\endaligned
$$
with $\nu \in \hat U$, $\nu \neq
{\nu}^{-1}$.
\end{theorem}

Using the Hecke algebra for the pair $(GL(2,F_q), K)$, we compute the eigenvalues of $\Delta$. Recall that the zonal spherical functions are simultaneously eigenfunctions and eigenvalues of invariant integral operators, or rather, Hecke  algebras, in the case of a finite group. The adjacency matrix is an element of the Hecke algebra for $(GL(2,F_q), K)$. A trivial extension of the Hecke algebra to a larger algebra permits the zonal spherical functions to be eigenfunctions and eigenvalues of $\Delta$ as well. We see that both algebraic and geometric methods are equally useful for an explicit description of the fundamental solution on $H_q$.
The following result is a consequence of all the previous theorems:

\begin{theorem}\label{Thm 4.4}
The fundamental solution to (\ref{eq-2})--(\ref{eq-3}) is
$$
E(t;r) = \sum_{\beta \in
\widehat{F_q^{\times}}}
         e^{-{\omega}^{\beta}(r)t}
{\omega}^{\beta}(r)
         +
         \sum_{\nu \in \hat U}
e^{-{\omega}^{\nu}(r)t}
{\omega}^{\nu}(r).
$$
\end{theorem}
To summarize, in analogy to the case of the geometric Laplacian on $H$, the combinatorial Laplacian on $H_q$ permits a solution to the heat equation that is constant on the $K$-double cosets, or rather, the $S_r$ subsets of $H_q$.
\section{Theta functions for $H_q$}
We conclude with the construction of finite analogues of theta functions and with a theta function correspondence between the heat equation and the combinatorial Laplacian of $H_q$.

To explicitly describe theta functions for $H_q$, we need to be more explicit regarding the set of non-decomposable characters
of $(F_q(\sqrt{\delta}))^{\times}$. If $q >
2$ and $\delta$ is
a nonsquare of $F_q$, then there exist maps
$\phi$,
$\tilde{\nu}$
such that
$$
(F_q(\sqrt{\delta}))^{\times}
\overset{\phi}\to\rightarrow
(Z_{q^2})^{\times}
\overset{\tilde{\nu}}\to\rightarrow
\mathbb{C}^{\times}.
$$
Let ${\tilde{\nu}}_j : (Z_{q^2})^{\times} \rightarrow \mathbb{C}^{\times}$ be defined by $$
{\tilde{\nu}}_j (k) = e^{\frac{2\pi
ijk}{q^2-1}}, \quad 1 \le k \le
q^2-1
$$
for fixed $j$ = 1, 2,$\dots$, or $q^2-1$. If
$\zeta$ is a generator of $(F_q(\sqrt{\delta}))^{\times}$, define
$\phi:(F_q(\sqrt{\delta}))^{\times}
\rightarrow (Z_{q^2})^{\times}$
by the isomorphism $\phi ({\zeta}^m) = m$, $1
\le m \le q^2-1$. A non-decomposable character
$\nu : (F_q(\sqrt{\delta}))^{\times}
\rightarrow \mathbb{C}$ can be defined
as the composition $\nu = \tilde{\nu} \circ
\phi$.

In other words,
$$
\aligned
{\nu}_j ({\zeta}^m) &= e^
                       {
                       \left(
                       \frac{2\pi
ij}{q^2-1}\phi ({\zeta}^m)
                       \right)
                       }
                       \\
                    &= e^
                       {
                       \left(
                       \frac{2\pi ijm}{q^2-1}
                       \right)
                       }
                     ,
\endaligned
$$
for ${\zeta}^m \in (F_q(\sqrt{\delta}))^{\times}$ and fixed
$j$ = 1, 2,$\dots$, or $q^2-1$.
We have that $E(t;r) = E^p(t;r) + E^c(t;r)$,  with $E^p(t;r)$ and
$E^c(t;r)$ the terms in the fundamental solution resulting from the principal and cuspidal zonal spherical functions, respectively.
Fix $\zeta$, generator of $(F_q(\sqrt{\delta}))^{\times}$, and $r \in F_q^{\times} -$ {1}.

Define
$$
\aligned
U    &= \{ m \in Z_{q^2}^{\times} \mid
N({\zeta}^m) = 1\}, \\
V(r) &= \{ y \in F_q^{\times} \mid
       x^2 = ry + \delta (y-1)^2, \text{ some } x
\in F_q \}.
\endaligned
$$

Note that $(F_q(\sqrt{\delta}))^{\times}
\supset U \supset
            F_q^{\times} \supset V(r)$.

Also, define
$$
\aligned
\mathcal{O}(r) &=
\{ m \in Z_{q^2}^{\times} \mid
Tr({\zeta}^m) - \frac{r+1}{r-1} =
   k^2, \text{ some } k \in F_q^{\times}
\}, \\
\mathcal{N}     &=
\{ m \in Z_{q^2}^{\times} \},
\endaligned
$$
and
$$
{\chi}_{\mathcal{A}} (m) =
\left\{
\aligned
&1, m\in \mathcal{A} \\
&0, \text{ otherwise }
\endaligned
\right.
$$ for a fixed set $\mathcal{A}$.
Therefore, the principal zonal spherical function takes the form:
$$
{\omega}^l(r) = \frac1{q+1}
                \sum_{m \in V(r)}
e^{\frac{2\pi ilm}{q-1}},
                \quad l \in F_q^{\times}.
$$

Set ${\lambda}_l^p (r) = (q+1){\omega}^l(r)$, the eigenvalue from the principal part. The superscript tells us to take ${\omega}^l$ to be the  principal zonal spherical function. Thus,
$$
E^p(t;r) = \frac1{q+1} \sum_{(l,m) \in
F_q^{\times} \times V(r)}
           e^{-{\lambda}_l^p (r)t + \frac{2\pi
ilm}{q-1}},
$$ which is the principal part of the solution.

Similarly, the cuspidal zonal spherical function takes the form:
$$
{\omega}^l(r) = \frac1{q+1}
                \sum_{m \in U}
                e^{2\pi i\big(
                \frac{{\chi}_{\mathcal{0} (r)} (m) +
{\chi}_{\mathcal{N}} (m)}{2}
                   + \frac{lm}{q^2-1}
                         \big)
}, \quad l \in Z_{q^2}^{\times}.
$$
Set ${\lambda}_l^c (r) = (q+1){\omega}^l(r)$, the eigenvalue from the cuspidal part. The superscript tells us to take
${\omega}^l$ to be the cuspidal zonal spherical function. Thus,
$$
E^c(t;r) = \frac1{q+1} \sum_{(l,m) \in
Z_{q^2}^{\times} \times U}
           e^{-{\lambda}_l^c (r)t +
             2\pi i\big(
             \frac{{\chi}_{\mathcal{O} (r)} (m) +
{\chi}_{\mathcal{N}} (m)}{2}
              + \frac{lm}{q^2-1}
                    \big)
},
$$
which is the cuspidal part of the solution. We have proved the following theorem.
\begin{theorem}\label{Thm 5.1}
The fundamental solution to (\ref{eq-2})--(\ref{eq-3}) has the explicit form:
\begin{equation}\label{eq-4}
E(t;r) = \frac1{q+1} \sum_{(l,m) \in
Z_{q^2}^{\times} \times U}
         e^{-{\alpha}_r(l)t + 2\pi
i{\beta}_r(l,m)},
\end{equation}
with
$$
{\alpha}_r(l) =
\left\{
\aligned
&{\lambda}_l^p (r) + {\lambda}_l^c (r),
\quad l \in F_q^{\times} \\
&{\lambda}_l^c (r), \quad l \in
F_{q^2}^{\times} - F_q^{\times},
\endaligned
\right.
$$
and
$$
2\pi i{\beta}_r(l,m) =
\left\{
\aligned
&\frac{{\chi}_{\mathcal{O} (r)} (m) +
{\chi}_{\mathcal{N}} (m)}{2}
 + \frac{lm(q+2)}{q^2-1}, \quad l \in
F_q^{\times},
 m \in V(r) \\
&\frac{{\chi}_{\mathcal{O} (r)} (m) +
{\chi}_{\mathcal{N}} (m)}{2}
 + \frac{lm}{q^2-1}, \quad l \in
F_{q^2}^{\times} - F_q^{\times},
 m \in U - V(r).
\endaligned
\right.
$$
\end{theorem}
If $\tau \in H$, the Poincar\'e upper half-plane ($Im (\tau) > 0$),
and $z \in \mathbb{C}$, then the analytic function
$$
\theta (z,\tau ) = \sum_{n \in Z}
e^{i\pi n^2\tau + 2\pi inz}
$$ is the fundamental periodic solution to the heat equation whenever
$\tau = it$, $t \in \mathbb{R}^{>0}$ \cite{Mum}. We say that
$$
\theta (z, it) = \sum_{n \in \mathbb{Z}}
e^{-\pi n^2t + 2\pi in z}
$$
is a theta function on the lattice $\mathbb{Z}$.

Thus (\ref{eq-4}) is a finite analogue of this classical function, and is thought of as the theta function on the finite lattice $F_{q^2}^{\times} \times U$ contained in the regular lattice
$F_{q^2}^{\times} \times F_{q^2}^{\times}$, connected to $H_q$ via the combinatorial Laplacian $\Delta$.
It would be an interesting problem to compare our results with those
of Chung and Yau \cite{Chu}.

\newpage
\bibliographystyle{plain}
\bibliography{heateqn}

\end{document}